\documentclass[12pt]{article}
\usepackage{fullpage}
\usepackage[latin1]{inputenc}
\usepackage{amsthm,amssymb,amsmath,amsfonts} 
\usepackage{enumitem}

\newtheorem{definition}{Definition}
\newtheorem{proposition}[definition]{Proposition}
\newtheorem{lemma}[definition]{Lemma}
\newtheorem{theorem}[definition]{Theorem}

\usepackage{color}

\newcommand{\subscript}[2]{$#1#2$}
\newcommand{\RR}{{\mathbb R}} 
\newcommand{\proscal}[2]{\left\langle#1\:,#2\right\rangle}  

\newcommand{\variableTime}{t}
\newcommand{\variableS}{s}

\newcommand{\state}{x}
\newcommand{\statebis}{\state'}

\newcommand{\STATE}{{\mathbb X}}
\newcommand{\CONTROL}{\mathbb{U}}
\newcommand{\control}{u}
\newcommand{\controlBis}{u'}
\newcommand{\ControlPaths}{\mathcal{U}}
\newcommand{\dynamicMapping}{f}
\newcommand{\dynamicMappingBis}{g}
\newcommand{\dynamicMappingTer}{h}
\newcommand{\cone}{K}
\newcommand{\DomainTime}{[0,+\infty)}

\newcommand{\DUAL}{\RR^n}
\newcommand{\dual}{y}
\newcommand{\Flow}[1]{\Psi_{#1}}
\newcommand{\flow}[3]{\Psi_{#2}(#1,#3)} 
\newcommand{\controlledFlow}[4]{\Psi_{#2}^{#4}(#1,#3)} 
\newcommand{\defined}{=}
\newcommand{\interior}{{\rm{int}}}
\newcommand{\anyVector}{v}
\newcommand{\derivativeTime}{\frac{d}{dt}}

\newcommand{\targetSet}{\mathbb{D}}
\newcommand{\viabilityKernel}{\mathbb{V}}

\newcommand{\polar}{^\oplus}  
\newcommand{\polarCone}{\cone\polar}
\newcommand{\orthogonal}{^\perp}
\newcommand{\subZero}{_0}

\newcommand{\WolbLarva}{L_W}
\newcommand{\WolbAdult}{A_W}
\newcommand{\NonWLarva}{L_U}
\newcommand{\NonWAdult}{A_U}

\renewcommand{\ll}{\prec\!\prec}
\newcommand{\nll}{\not\ll}

\newcommand{\eqsepv}{\; , \enspace}       
\newcommand{\eqfinv}{\; ,}                
\newcommand{\eqfinp}{\; .}

\newcommand{\np}[1]{(#1)}                                  
\newcommand{\bp}[1]{\big(#1\big)}                          
\newcommand{\Bp}[1]{\Big(#1\Big)}

\begin{document}

\title{Comparison Theorem for Viability Kernels\\via Conic Preorders}

\author{Michel De Lara\footnote{%
CERMICS, Ecole des Ponts, Marne-la-Vall\'ee, France}, 
Pedro Gajardo\footnote{%
Departamento de Matem\'atica, Universidad T\'ecnica Federico Santa Mar\'ia,
Avenida Espa\~{n}a 1680, Valpara\'iso, Chile}, 
Diego Vicencio$^\dagger$}

\maketitle

\begin{abstract}
In natural resource management, decision-makers often aim at maintaining the
state of the system within a desirable set for all times.
For instance, fisheries management procedures include keeping the
spawning stock biomass over a critical threshold.
Another example is given by the peak control of an epidemic outbreak
that encompasses maintaining the
number of infected individuals below medical treatment capacities.
In mathematical terms, one controls a dynamical system.
Then, keeping the state of the system within a desirable set for all times is 
possible when the initial state belongs to the so-called viability
kernel. We introduce the notion of conic quasimonotonicity reducibility.
With this property, we provide a comparison theorem by inclusion between two viability
kernels, corresponding to two control systems in the infinite horizon case. We
also derive conditions for equality. We illustrate the method with a model for the
biocontrol of a vector-transmitted epidemic. 
\end{abstract}





\section{Introduction}\label{sec:intro}

In natural resource management, one often aims at maintaining the state of
the system within a desirable set for all times like, for instance, 
spawning stock biomass over a critical threshold in fishery management
\cite{bene2001viability,ICES,eisenack2006viability,Krawczyketal:2013,MR2214187}, 
number of infected individuals below a health threshold in epidemic
control (see \cite{DeLara2016}, and the concept of \emph{endemic channel} in
\cite{Brady-Smith-Scott-Hay:2015} and in the Operational Guide of the World
Health Organization \cite{WHO:2017}),  
population abundance above extinction level in population viability analysis \cite{Williams-Nichols-Conroy:2002}.
This is possible when the
initial state belongs to the so-called \emph{viability kernel}
\cite{aubin1990,Aubin2011}. 
The viability approach --- consisting in characterizing, computing or estimating
the  corresponding viability kernel --- 
has notably been applied to the analysis of topics in natural resource management, as recently reviewed in \cite{zaccour}.

In the literature of the last decades, one can find several methods for the
challenging task of computing the viability kernel (see for instance,
\cite{Aubin2011,Bonneuil:2006,Deffuant-Chapel-Martin:2007,DLD,monotone,DGR,MAIDENS20132017,GH2019_vk,Krawczyketal:2013,YOUSEFI201913,SP94}). In
general, numerical methods for such computation can be
implemented only for systems with a few number of state variables 
or, as in \cite{Bonneuil:2006}, for a limited time horizon. This is
because of the so-called curse of dimensionality. 
This is an important drawback in the study of
some natural resource management or epidemic control problems that have models
composed of many state variables, such as age-structured fish-stock
population models which often display more
than ten state variables (see \cite{caswell,DLD,ICES,QD}).

To overcome the curse of dimensionality, some approaches make use of 
linearity \cite{MAIDENS20132017} or of monotonicity properties
induced by the positive orthant in the state space
\cite{monotone,DGR,DeLara2016}.
In this work, we aim at obtaining a characterization of the viability kernel 
of controlled systems in the infinite horizon case
under monotonicity properties, but in a broad sense, namely induced by a 
so-called conic preorder. For this purpose, we consider a convex cone~$\cone$ in the state space
and the induced conic preorder~$\preceq_\cone$.
Our key assumption is that the dynamics defining the
system under study is $\cone$-quasimonotone, a generalization of the
cooperativeness property for dynamical systems. Under this assumption, 
we can establish a comparison theorem for the solutions
of the underlying differential equation \cite{Hirsch2004,Hirsch2003}. 
Our main contribution relies on a second assumption: 
the existence of a  \emph{reduction} of the
controls (to be explained later) associated with the convex cone~$\cone$. The idea of the
reduction is that, given a control path and the associated state
path, one can find another control path (ideally in a reduced control
space) whose associated state path is preordered with respect to the first
one. We prove that the problem of computing the viability kernel can be carried out by
exploring a smaller set of paths, hence reducing the complexity of the
problem. 

The paper is organized as follows.
In Sect.~\ref{Controlled_dynamical_systems_and_viability_kernels}, 
we present the main definitions regarding 
controlled dynamical systems and viability kernels.
Then, in Sect.~\ref{Comparison_of_viability_kernels_via_conic_orders},
we introduce conic preorders and prove our main result, that is,
a comparison theorem for viability kernels.
 Finally, Sect.~\ref{Application_to_viable_control_of_the_Wolbachia_bacterium}
is devoted to an illustration in the
biocontrol of a vector-transmitted epidemic.

\section{Controlled dynamical systems and viability kernels}
\label{Controlled_dynamical_systems_and_viability_kernels}

In~\S\ref{Controlled_dynamical_systems}, 
we present controlled dynamical systems and,
in~\S\ref{Desirable_set_and_viability_kernel},
the viability kernel associated with a controlled dynamical system
and a desirable set. 

\subsection{Controlled dynamical systems}
\label{Controlled_dynamical_systems}

We give a formal definition of controlled dynamics
and controlled dynamical systems,
including technical assumptions that will be useful in the paper. 
We consider \( \RR^{n} \) for \emph{state space}
and \( \RR^{m} \) for \emph{control space}, where $n$ and $m$ are positive integers.

\begin{definition}
A \emph{controlled dynamics} is a mapping
\( \dynamicMapping:\STATE \times \CONTROL \to \RR^{n} \),
where \( \STATE \subset \RR^{n} \) is a closed subset of~\( \RR^{n} \),
and \( \CONTROL \subset \RR^{m} \) is a (Borel) measurable subset of~\( \RR^{m} \),
with the following two properties:
\( \dynamicMapping \) is jointly measurable in the state and control variables;
\( \dynamicMapping \) is locally Lipschitz in the state variable 
uniformly in the control variable, that is, 
for every \(\state_0 \in \STATE\) there exists \(L>0\) and \(\delta>0\) such
that, for any \( \state, \statebis \in \STATE  \), 
\[
\|\state-\state_0\| \leq \delta 
\text{ and }
\|\statebis-\state_0\| \leq \delta \implies
\|\dynamicMapping(\state,\control)-\dynamicMapping(\statebis,\control)\| \leq 
L \|\state - \statebis\| \eqsepv \forall \control \in \CONTROL
\eqfinv
\]
where \(\|\cdot\|\) is any norm on~\(\RR^{n}\).
\label{de:controlled dynamics}
 \end{definition}

We define the \emph{set of (admissible) control paths} by 
\begin{equation}
\ControlPaths =\{ \control\np{\cdot}:\DomainTime\to\CONTROL
\mid \control\np{\cdot} \textrm{ is measurable} \} 
\eqfinp  
\label{eq:control_paths}
\end{equation}
Given a controlled dynamics $\dynamicMapping:\STATE \times
\CONTROL \to \RR^{n}$, as in Definition~\ref{de:controlled dynamics},
and a control path \( \control\np{\cdot}:\DomainTime\to\CONTROL \)
as in~\eqref{eq:control_paths}, it can be shown
(see \cite[Theorem~7.4.1, p.~263]{Clarke:1983} or \cite[Theorem~1.1, p.~178]{CLSW:1998})
that the differential equation 
\begin{equation}
    \dot{\state}(\variableTime) = 
\dynamicMapping\bp{\state(\variableTime),\control(\variableTime)} 
\eqsepv \state(0)=\state_0 \in \STATE \text{ given}
\label{eq:controlledModel}
\end{equation}
has a unique solution defined on an (open to the right) time interval 
\( [0,T) \subset \DomainTime \).
When this unique solution is defined for all \( t \in \DomainTime \),
we denote it by \( \state(t) =
\controlledFlow{\variableTime}{\dynamicMapping}{\state\subZero}{\control(\cdot)}
\), that is,
\begin{equation}
\state(t) =
\controlledFlow{\variableTime}{\dynamicMapping}{\state\subZero}{\control(\cdot)} 
\iff
\dot{\state}(\variableTime) = 
\dynamicMapping\bp{\state(\variableTime),\control(\variableTime)} 
\eqsepv \state(0)=\state_0 
\eqfinv
\label{eq:flow}
 \end{equation}
and we call the mapping~\( \Psi \)
the \emph{flow} of the \emph{controlled dynamical
  system}~\eqref{eq:controlledModel}. 
We also say that the controlled dynamics~$\dynamicMapping$
generates a \emph{global flow}.

\subsection{Desirable set and viability kernel}
\label{Desirable_set_and_viability_kernel}

In \emph{viability theory}, one aims to determine a set of
initial conditions which allow to keep the state and control of a
dynamical system inside a so-called desirable set
by means of suitable control paths  \cite{aubin1990,Aubin2011}.

\begin{definition}
Let ~\( \dynamicMapping:\STATE \times \CONTROL \to \RR^{n} \) be a given controlled dynamics
as in Definition~\ref{de:controlled dynamics}, 
and suppose that it generates a global flow~\( \Psi \). 
Given a subset 
\begin{equation*}
\targetSet \subset \STATE\times\CONTROL
\eqfinv  
\end{equation*}
called the \emph{desirable set}, 
we define the \emph{viability kernel}, associated 
with the controlled dynamics~$\dynamicMapping$
and with the desirable set~$\targetSet$,
by
\begin{equation}
\label{eq:viability_kernel}
\begin{split}
\viabilityKernel\np{\dynamicMapping,\targetSet} =\\
\left\{\state\subZero\in\STATE
\mid \exists \control(\cdot) \in \ControlPaths \eqsepv 
\Bp{ { \controlledFlow{\variableTime}{\dynamicMapping}{\state\subZero}{\control(\cdot)},
\control(\variableTime) } } \in \targetSet 
\eqsepv \forall \variableTime \in \DomainTime \right\} 
\eqfinp
\end{split}
\end{equation}
\label{de:viability_kernel}
\end{definition}
Thus, the viability kernel represents the set of initial conditions 
\( \state\subZero\in\STATE \)
such that there exists a control path $\control(\cdot) \in \ControlPaths $ in~\eqref{eq:control_paths}
for which the associated state and control paths, generated by~\eqref{eq:controlledModel},
remain in the desirable set~$\targetSet$ for all times. 

For a decision maker, knowing the viability kernel has practical
interest since it describes the set of states from which controls can be
found that maintain the system in a desirable configuration
forever. Nevertheless, computing this kernel is not an easy task in general.
However, under additional assumptions on the dynamics and on the desirable set,
it is possible to simplify the computation 
as we show in the following section.

\section{Comparison theorem for viability kernels via conic preorders}
\label{Comparison_of_viability_kernels_via_conic_orders}

Now, we present our main result, which is a comparison theorem for viability
kernels by means of so-called conic preorders.
In~\S\ref{Conic_orders_and_conic_quasimonotonicity}, 
we recall the notions of conic preorder and of conic quasimonotonicity.
Then, we propose the new definition of 
conic quasimonotonicity reducibility for controlled dynamical systems
in~\S\ref{Conic_quasimonotonicity_for_controlled_dynamical_systems}.
Thus equipped, we state our main result 
in~\S\ref{Comparison_of_viability_kernels}.

\subsection{Conic preorders and conic quasimonotonicity}
\label{Conic_orders_and_conic_quasimonotonicity}


Let $\cone \subset \RR^{n}$ be a convex cone,
that is, $\alpha \cone \subset \cone$ for all $\alpha \in \RR_+$ (hence \( 0 \in \cone \)), 
and $\cone + \cone \subset \cone$. 
It is well-known \cite{Angeli,Hirsch2004,Hirsch2003,smith1995monotone}
that such a convex cone induces a preorder 
(that is, a transitive and reflexive relation) on~$\RR^{n}$,
denoted by~$\preceq_\cone$ and given by
\begin{equation}
\bp{ \forall \state,\statebis \in \RR^{n} } \qquad 
  \state \preceq_\cone \statebis
\iff 
\statebis - \state \in \cone
\eqfinp 
\label{eq:K-order}
\end{equation}
%
%
Let \( \proscal{\cdot}{\cdot} \) stands for the usual inner product in
\(\RR^{n}\).
The \emph{dual cone} associated with the
cone~$\cone$ is \cite{Bauschke2011,Hirsch2003}
\begin{equation}
\polarCone = \left\{ \dual \in  \DUAL \mid 
\proscal{\state}{\dual} \ge 0 \eqsepv
\forall \state \in \cone \right\} 
\eqfinp
\label{eq:positive_polar_cone}
\end{equation}
The following definition is introduced in \cite{Hirsch2004,Hirsch2003}.
We reframe it with our own notations.

\begin{definition}[\cite{Hirsch2004,Hirsch2003}]
We say that a \emph{mapping} (dynamics)
\( \dynamicMappingTer:\STATE \times \DomainTime\to\RR^{n} \),
where \( \STATE \subset \RR^{n} \) is a closed subset of~\( \RR^{n} \),
is \emph{$\cone$-quasimonotone} if the following condition holds 
\begin{multline}
\label{eq:quasimonotoneCondition}
\bp{ \forall \state,\statebis \in \STATE \eqsepv 
\forall \dual \in \polarCone}\\
 \state \preceq_{\cone\cap \left\{\dual \right\}\orthogonal} 
\statebis \implies \dynamicMappingTer(\state,\variableTime) 
\preceq_{\left\{\dual\right\}\polar}
\dynamicMappingTer(\statebis,\variableTime)  \eqsepv 
\forall \variableTime \in \DomainTime 
\eqfinp 
\end{multline}
\label{de:K-quasimonotone_dynamics}
\end{definition}
In this definition, we use the preorders given by the convex cones $\cone\cap
\left\{\dual\right\}\orthogonal$ and $\left\{\dual\right\}\polar$, 
where \(\left\{\dual\right\}\orthogonal\) denotes the orthogonal space to~\(\dual\). 
From the definition~\eqref{eq:K-order} of the preorder induced by a convex cone, 
we obtain that 
\begin{subequations}
  \begin{align}
\state \preceq_{\cone\cap \left\{\dual\right\}\orthogonal}     \statebis  
&\iff 
\statebis - \state \in \cone 
\text{ and } 
\proscal{\statebis-\state}{\dual} = 0 
\eqfinv    
\\
\state \preceq_{\left\{\dual\right\}\polar} \statebis  
&\iff
    \proscal{\statebis-\state}{\dual} \ge 0 
\eqfinp
  \end{align}
\end{subequations}

When the mapping~$\dynamicMappingTer$ displays additional
regularity properties and the cone $\cone$ is one of the orthants in $\RR^n$, there exists a more amenable characterization of
$\cone$-quasimonotonicity,
as presented in the next proposition~\cite{smith1995monotone}.

\begin{proposition}(\cite[Proposition 5.1]{smith1995monotone})
If the mapping \( \dynamicMappingTer:\STATE \times \DomainTime\to\RR^{n} \)
in Definition~\ref{de:K-quasimonotone_dynamics} 
is differentiable with respect to the first variable,
where \( \STATE \subset \RR^{n} \) is the closure of an open subset of~\( \RR^{n} \),
and if the cone~$\cone$ is one of the orthants of $\RR^n$, that is
\[
\cone=\{ (\state_1,\ldots,\state_n) \in \RR^n 
\mid (-1)^{m_j}\state_j \geq 0 \eqsepv j=1,\ldots,n\} 
\eqfinv
\]
where $(m_1,\ldots,m_n) \in \{0,1\}^n$, 
then the mapping $\dynamicMappingTer$ is
$\cone$-quasimonotone if and only if  
\begin{equation}
 (-1)^{m_i+m_j}\frac{\partial \dynamicMappingTer_i}{\partial
  \state_j}(\state,\variableTime) \geq 0 
\eqsepv \forall i \neq j \eqsepv \forall (\state,\variableTime) \in
\STATE \times \DomainTime 
\eqfinp
\label{eq:cond-orthant}
\end{equation}
\label{prop:orthants}
\end{proposition}
When the convex cone is the positive orthant~\(\cone=\RR_+^{n}\), 
condition~\eqref{eq:cond-orthant}  is called \emph{cooperativeness} in
\cite{smith1995monotone},
as it reads 
\( \frac{\partial \dynamicMappingTer_i}{\partial
  \state_j}(\state,\variableTime) \ge 0 
\eqsepv \forall i \neq j \eqsepv \forall (\state,\variableTime) \in
\STATE \times \DomainTime \).

\subsection{Conic quasimonotonicity reducibility for controlled dynamical systems}
\label{Conic_quasimonotonicity_for_controlled_dynamical_systems}

The following definition is new.

\begin{definition}
Let $\cone \subset \RR^{n}$ be a convex cone.
Let $\phi:\CONTROL\to\CONTROL$ be a measurable mapping.
We say that a controlled dynamics
$\dynamicMapping:\STATE \times \CONTROL \to \RR^{n}$,
as in Definition~\ref{de:controlled dynamics}, 
is \emph{\( \np{\cone,\phi} \)-quasimonotone reducible} if the two following conditions hold.
\begin{enumerate}[label=(\subscript{H}{{\arabic*}})]
\item
\label{it:quasimonotone_controlled_dynamics_a}
For all control path~$\control(\cdot) \in \ControlPaths$ in~\eqref{eq:control_paths}, 
the mapping \( {\dynamicMappingTer}_{\control\np{\cdot}}
: \STATE \times \DomainTime \to \RR^{n} \) 
defined by 
${\dynamicMappingTer}_{\control\np{\cdot}}(\state,\variableTime) \defined \dynamicMapping\bp{\state,\control(\variableTime)}$ 
is $\cone$-quasimonotone (as in Definition~\ref{de:K-quasimonotone_dynamics}).
\item
\label{it:quasimonotone_controlled_dynamics_b}
The measurable mapping $\phi:\CONTROL\to\CONTROL$ 
has the property that
\begin{equation}
\label{eq:controlledInequality}
	\dynamicMapping\np{\state,\control}
\preceq_\cone 
\dynamicMapping\bp{ \state,\phi\np{\control} }
\eqsepv 
\forall (\state,\control) \in \STATE\times\CONTROL
\eqfinp
\end{equation}
The measurable mapping~$\phi$ is called a \emph{$\cone$-reduction} 
for the controlled dynamics~$\dynamicMapping$.
\end{enumerate}
\label{de:quasimonotone_controlled_dynamics}
\end{definition}
The notion of $\cone$-reduction is interesting in practice if
the mapping $\phi:\CONTROL\to\CONTROL$ is not surjective
(that is\footnote{%
    $J \subsetneq L$ stands for $J\subset L$ and $J\not=L$.}, \(\phi(\CONTROL) \subsetneq \CONTROL\)),
and more precisely if its image \(\phi(\CONTROL) \) is ``small'' 
because, in some way, we are reducing the control space~\(\CONTROL\).
The following result provides a sufficient condition to compare flows of
controlled dynamics, based on \( \np{\cone,\phi} \)-quasimonotone reducibility.
 
\begin{proposition}
\label{pr:comparisonControlled}
Let $\cone \subsetneq \RR^{n}$ be a closed convex cone with nonempty interior and 
 $\phi:\CONTROL\to\CONTROL$  a measurable mapping. Let ~\( \dynamicMapping:\STATE \times \CONTROL \to \RR^{n} \) be a given controlled dynamics
as in Definition~\ref{de:controlled dynamics}, 
and suppose that it generates a global flow~\( \Psi \),
and that it is \( \np{\cone,\phi} \)-quasimonotone reducible, as in
Definition~\ref{de:quasimonotone_controlled_dynamics}.

Then, for any control path $\control(\cdot) \in \ControlPaths$ in~\eqref{eq:control_paths}, we have that 
\begin{equation}
  \begin{split}
\state\subZero,\statebis\subZero \in \STATE
\text{ and }
\state\subZero \preceq_\cone \statebis\subZero 
\implies \\
\controlledFlow{\variableTime}{\dynamicMapping}{\state\subZero}{\control(\cdot)} 
\preceq_\cone 
\controlledFlow{\variableTime}{\dynamicMapping}{\statebis\subZero}{\control_\phi(\cdot)} 
\eqsepv \forall \variableTime \in \DomainTime 
\eqfinv    
  \end{split}
\label{eq:monotonicityControlledFlow}
\end{equation}
where the reduced control path \( \control_\phi(\cdot)  \in \ControlPaths \)
is defined by 
\begin{equation}
  \control_\phi(t) = \phi\bp{\control(t)} \eqsepv 
\forall t \in \DomainTime 
\eqfinp
\label{eq:reduced_control_path}
\end{equation}
\end{proposition}
  
\begin{proof}
The control path \( \control_\phi(\cdot) \),
defined by \( \control_\phi(t) = \phi\bp{\control(t)} \) 
for all \( t \in \DomainTime \), is measurable 
as both \( \control(\cdot) \) and \( \phi \) are measurable mappings. 
For a control path $\control(\cdot) \in \ControlPaths$ in~\eqref{eq:control_paths}, 
we define the dynamics mappings 
\(\dynamicMappingTer_{\control\np{\cdot}}:\STATE \times \DomainTime\to\RR^{n}\)
and
\(\dynamicMappingTer_{\control_\phi\np{\cdot}}:\STATE \times \DomainTime\to\RR^{n}\)
by
\begin{equation*}
\dynamicMappingTer_{\control\np{\cdot}}(\state,\variableTime) =
\dynamicMapping\bp{\state,\control(\variableTime)}
\text{ and } 
\dynamicMappingTer_{\control_\phi\np{\cdot}}(\state,\variableTime) =
\dynamicMapping(\state,\control_\phi(\variableTime)) 
\eqsepv \forall (\state,\variableTime) \in \STATE \times \DomainTime
\eqfinp
\end{equation*}
By assumption~\ref{it:quasimonotone_controlled_dynamics_a} 
in Definition~\ref{de:quasimonotone_controlled_dynamics},
the dynamics mapping~${\dynamicMappingTer}_{\control\np{\cdot}}$
is $\cone$-quasimonotone.
By assumption~\ref{it:quasimonotone_controlled_dynamics_b} 
in Definition~\ref{de:quasimonotone_controlled_dynamics},
Equation~\eqref{eq:controlledInequality} gives that 
\[
{\dynamicMappingTer}_{\control\np{\cdot}}(\state,\variableTime) \preceq_\cone 
{\dynamicMappingTer}_{\control_\phi\np{\cdot}}(\state,\variableTime) 
\eqsepv \forall (\state,\variableTime)\in\STATE\times\DomainTime
\eqfinp
\]
Looking at the assumptions of Lemma~\ref{lem:comparison},
in~\ref{Appendix},
we can check that they are all satisfied.
The result~\eqref{eq:monotonicityControlledFlow} follows directly. 
\end{proof}

Let us contrast the assumptions in Proposition~\ref{pr:comparisonControlled}
with the following ones, given in \cite{Angeli}:
\begin{enumerate}[label=(\subscript{\hat{H}}{{\arabic*}})]
    \item 
\label{condition:3} 
The control set \(\CONTROL\) is a convex subset of $\RR^{m}$
and there exists a preorder $\preceq_{\cone_\CONTROL}$ given by a closed convex cone
$\cone_\CONTROL\subset \RR^{m}$;
    \item 
\label{condition:4}
For all $\control(\cdot) \in \ControlPaths$ in~\eqref{eq:control_paths}, the mapping
$(\state,\variableTime)\to\dynamicMapping(\state,\control(\variableTime))$
is $\cone$-quasimonotone;
    \item 
\label{condition:5} 
For any control paths $\control(\cdot),\controlBis(\cdot) \in \ControlPaths$, 
as in~\eqref{eq:control_paths}, one has that, if $\control(\variableTime) \preceq_{\cone_\CONTROL}
\controlBis(\variableTime)$ for all $\variableTime \in \DomainTime$, then
$\dynamicMapping(\state,\control(\variableTime)) \preceq_\cone
\dynamicMapping(\state,\controlBis(\variableTime))$ for all $\state \in
\STATE$ and $\variableTime \in \DomainTime$.  
\end{enumerate}
The result in~\cite{Angeli} is a particular case of our result,
because the assumptions
\ref{condition:3}, \ref{condition:4} and \ref{condition:5} imply our assumptions
\ref{it:quasimonotone_controlled_dynamics_a} and
\ref{it:quasimonotone_controlled_dynamics_b} 
in Definition~\ref{de:quasimonotone_controlled_dynamics}. 
Indeed, first, condition \ref{it:quasimonotone_controlled_dynamics_a} is
the same as \ref{condition:4}.
Second, by taking for $\cone$-reduction mapping
any measurable mapping~$\phi: \CONTROL\to\CONTROL $ such that 
\( \phi(\control) \in (\control + \cone_\CONTROL)\cap \CONTROL \), 
we see that conditions \ref{condition:3} and \ref{condition:5} imply
\ref{it:quasimonotone_controlled_dynamics_a}. 
Therefore, to obtain the monotonicity property~\eqref{eq:monotonicityControlledFlow}, 
it is not necessary to have a preorder 
defined on the control space~\(\RR^{m}\) as in condition~\ref{condition:3}
in~\cite{Angeli}, but it is
enough to find a $\cone$-reduction 
as in condition~\ref{it:quasimonotone_controlled_dynamics_b}.

\subsection{Comparison theorem for viability kernels}
\label{Comparison_of_viability_kernels}

Now, we are ready to provide a comparison result for viability kernels, the main purpose of this work. 

\begin{theorem}
\label{theo:viability2}
Let $\cone \subsetneq \RR^{n}$ be a closed convex cone with nonempty interior and  $\phi:\CONTROL\to\CONTROL$  a measurable mapping.
Let ~\( \dynamicMapping:\STATE \times \CONTROL \to \RR^{n} \) be  a given  controlled dynamics
as in Definition~\ref{de:controlled dynamics}, 
and suppose that it generates a global flow~\( \Psi \),
and that it is \( \np{\cone,\phi} \)-quasimonotone reducible, as in
Definition~\ref{de:quasimonotone_controlled_dynamics}.
Let $\targetSet \subset \STATE\times\CONTROL$ be a desirable set.
We introduce
\begin{enumerate}
\item 
the \emph{reduced controlled dynamics}
$\dynamicMapping_\phi:\STATE \times \CONTROL \to \RR^{n}$
defined by
\begin{equation}
\dynamicMapping_\phi\np{\state,\control} = 
\dynamicMapping\bp{\state,\phi\np{\control}}
\eqsepv \forall  (\state,\control) \in \STATE\times\CONTROL
\eqfinv
\label{eq:reduced_controlled_dynamics}
\end{equation}
\item 
the \emph{extended desirable set}
\( \targetSet_\cone \subset \STATE \times \CONTROL \)
defined by
\begin{equation}
\label{eq:newTargetSet}
    \targetSet_\cone =  \targetSet + \bp{\cone \times \{0\}}
\eqfinp
\end{equation}
\end{enumerate}
Then, we have the following inclusion of viability kernels:
\begin{equation}
\viabilityKernel(\dynamicMapping,\targetSet) \subset 
\viabilityKernel(\dynamicMapping_\phi,\targetSet_\cone)
\eqfinp
\label{eq:inclusion_of_viability_kernels}
\end{equation}
Furthermore, if
\begin{equation}
\bigcup_{(\state,\control) \in \targetSet} (\state+\cone)\times \phi(\control) \subset \targetSet
\eqfinv
\label{eq:equality_of_viability_kernels_assumption}
\end{equation}
then we have the following equality between viability kernels:
\begin{equation}
\viabilityKernel(\dynamicMapping,\targetSet) = 
\viabilityKernel(\dynamicMapping_\phi,\targetSet_{\cone})
\eqfinp
\label{eq:equality_of_viability_kernels}
\end{equation}
\end{theorem}

\begin{proof} 
It is easily checked that the mapping~\( \dynamicMapping_\phi \)
in~\eqref{eq:reduced_controlled_dynamics} indeed is 
a controlled dynamics as in Definition~\ref{de:controlled dynamics}.
Moreover, it is immediate, 
from definition~\eqref{eq:reduced_control_path} of 
the reduced control path \( \control_\phi(\cdot) \) and 
from definition~\eqref{eq:flow} of the flow, that 
\begin{equation}
\controlledFlow{\variableTime}{\dynamicMapping}{\state\subZero}{\control_\phi(\cdot)}
=
\controlledFlow{\variableTime}{\dynamicMapping_\phi}{\state\subZero}{\control(\cdot)}
\eqsepv \forall \np{\variableTime,\state\subZero} \in \DomainTime\times\STATE
\eqfinp 
\label{eq:equality_of_flows}
\end{equation}
\medskip

\noindent $\bullet$
First, we prove the inclusion~\eqref{eq:inclusion_of_viability_kernels}, 
that is, $\viabilityKernel(\dynamicMapping,\targetSet) \subset
\viabilityKernel(\dynamicMapping_\phi,\targetSet_\cone)$. 
For this purpose, we consider 
$\state\subZero \in \viabilityKernel(\dynamicMapping,\targetSet)$,
and we show that $\state\subZero \in 
\viabilityKernel(\dynamicMapping_\phi,\targetSet_\cone)$. 

By definition~\eqref{eq:viability_kernel} of the 
viability kernel~\( \viabilityKernel(\dynamicMapping,\targetSet) \),
 there exists a control path $\control(\cdot) \in \ControlPaths$ in~\eqref{eq:control_paths} such that 
\[
    \bp{
      \controlledFlow{\variableTime}{\dynamicMapping}{\state\subZero}{\control(\cdot)},
\control(t) } 
    \in \targetSet\eqsepv \forall \variableTime \in \DomainTime 
\eqfinp
\]
As the assumptions of Proposition~\ref{pr:comparisonControlled} are
satisfied, Equation~\eqref{eq:monotonicityControlledFlow} gives
\[ 
\controlledFlow{\variableTime}{\dynamicMapping}{\state\subZero}{\control(\cdot)} 
\preceq_\cone 
\controlledFlow{\variableTime}{\dynamicMapping}{\state\subZero}{\control_\phi(\cdot)}
\eqsepv \forall \variableTime \in \DomainTime 
\eqfinp 
\]
Thus, from~\eqref{eq:equality_of_flows}, we deduce that
\begin{equation}
\label{eq:monotonicityControlledFlowPhi}
    \controlledFlow{\variableTime}{\dynamicMapping}{\state\subZero}{\control(\cdot)}
    \preceq_\cone
    \controlledFlow{\variableTime}{\dynamicMapping_\phi}{\state\subZero}{\control(\cdot)}
    \eqsepv \forall \variableTime \in \DomainTime 
\eqfinp 
\end{equation}
Then, we write 
\[
  \controlledFlow{\variableTime}{\dynamicMapping_\phi}{\state\subZero}{\control(\cdot)}    
=
    \controlledFlow{\variableTime}{\dynamicMapping}{\state\subZero}{\control(\cdot)}
    +
\overbrace{ \bp{
    \controlledFlow{\variableTime}{\dynamicMapping_\phi}{\state\subZero}{\control(\cdot)}
      -
      \controlledFlow{\variableTime}{\dynamicMapping}{\state\subZero}{\control(\cdot)}
    } }^{\in \cone }
\eqfinv 
\]
where the second term \(\bp{
    \controlledFlow{\variableTime}{\dynamicMapping_\phi}{\state\subZero}{\control(\cdot)}
      -
      \controlledFlow{\variableTime}{\dynamicMapping}{\state\subZero}{\control(\cdot)}}\) 
belongs to~\( \cone \) by~\eqref{eq:monotonicityControlledFlowPhi}
and by the definition~\eqref{eq:K-order} of the preorder~$\preceq_\cone$.
Therefore, from definition~\eqref{eq:newTargetSet} of 
\( \targetSet_\cone = \targetSet + \bp{\cone \times \{0\}} \), we
deduce that, for all $\variableTime \in \DomainTime$,
\[
  \bp{\controlledFlow{\variableTime}{\dynamicMapping_\phi}{\state\subZero}{\control(\cdot)},
    \control(\variableTime)} 
=
  \underbrace{ \bp{
  \controlledFlow{\variableTime}{\dynamicMapping}{\state\subZero}{\control(\cdot)},
    \control(\variableTime)} }_{\in \targetSet}   
+ 
  \underbrace{ \bp{
    \controlledFlow{\variableTime}{\dynamicMapping_\phi}{\state\subZero}{\control(\cdot)}
      -
      \controlledFlow{\variableTime}{\dynamicMapping}{\state\subZero}{\control(\cdot)}, 0} }_{\in \cone \times \{0\}}
\in \targetSet_\cone 
\eqfinp
  \]
This implies that $\state\subZero \in
\viabilityKernel(\dynamicMapping_\phi,\targetSet_\cone)$,
hence the first part of the proof is complete.
\medskip

\noindent $\bullet$
Second, we suppose that~\eqref{eq:equality_of_viability_kernels_assumption}
holds true and we
prove the equality~\eqref{eq:equality_of_viability_kernels}, that is, 
\( \viabilityKernel(\dynamicMapping,\targetSet) = 
\viabilityKernel(\dynamicMapping_\phi,\targetSet) \). 
By the just proven inclusion~\eqref{eq:inclusion_of_viability_kernels}, 
it suffices to show the reverse inclusion, that is, 
 \(\viabilityKernel(\dynamicMapping_\phi,\targetSet_\cone) \subset 
\viabilityKernel(\dynamicMapping,\targetSet)\). 
For this purpose, we consider 
\( \state\subZero \in \viabilityKernel(\dynamicMapping_\phi,\targetSet_{\cone}) \)
and we show that \( \state\subZero \in 
\viabilityKernel(\dynamicMapping,\targetSet)\). 

By definition~\eqref{eq:viability_kernel} of the 
viability kernel~\( \viabilityKernel(\dynamicMapping_\phi,\targetSet_{\cone}) \),
there exists a control path $\control(\cdot) \in \ControlPaths$ in~\eqref{eq:control_paths} such that 
\(
    \bp{
      \controlledFlow{\variableTime}{\dynamicMapping_\phi}{\state\subZero}{\control(\cdot)},
\control(t) } 
    \in \targetSet_{\cone}\), \( \forall \variableTime \in \DomainTime \).
From~\eqref{eq:equality_of_flows}, we deduce that
\[
\bp{ \controlledFlow{\variableTime}{\dynamicMapping}{\state\subZero}{\control_\phi(\cdot)},
\control(t) }
\in \targetSet_{\cone} \eqsepv \forall \variableTime \in \DomainTime 
\eqfinp 
\]
Now, by definition~\eqref{eq:newTargetSet} of~\( \targetSet_{\cone} \), 
for all \( \variableTime \in \DomainTime \)
there exist \( v_t \in \RR^n \) and \( w_t \in \RR^n \) such that
\begin{equation}
\controlledFlow{\variableTime}{\dynamicMapping}{\state\subZero}{\control_\phi(\cdot)}=v_t+w_t 
\eqsepv \bp{v_t,\control(t)} \in \targetSet
\eqsepv  w_t \in \cone 
\eqfinp  
\label{eq:vtandwt}
\end{equation}
We now show that the control path \( \control_\phi(\cdot) \)
in~\eqref{eq:reduced_control_path}
maintains the state and control 
\( \bp{ \controlledFlow{\variableTime}{\dynamicMapping}{\state\subZero}{\control_\phi(\cdot)},
\control_\phi(t) } \) in \( \targetSet\). Indeed, we have
\begin{align*}
\bp{ \controlledFlow{\variableTime}{\dynamicMapping}{\state\subZero}{\control_\phi(\cdot)},
\control_\phi(t) } 
&=
\bp{  \controlledFlow{\variableTime}{\dynamicMapping}{\state\subZero}{\control_\phi(\cdot)},
\phi(\control(t)) }
\tag{by definition~\eqref{eq:reduced_control_path} of \( \control_\phi(\cdot) \)}
\\
&=
   \bp{ v_t+w_t,\phi(\control(t))} 
\tag{by definition~\eqref{eq:vtandwt} of~$v_t$ and $w_t$}
\\
&\in 
\bigcup_{(\state',\control') \in \targetSet} (\state'+\cone)\times \phi(\control') 
\tag{as \((v_t,\control(t)) \in \targetSet\) and \( w_t \in \cone \) by~\eqref{eq:vtandwt}}
\\
&\in
\targetSet
\tag{by~\eqref{eq:equality_of_viability_kernels_assumption}} \eqfinp
\end{align*}
By definition~\eqref{eq:viability_kernel} of the 
viability kernel~\( \viabilityKernel(\dynamicMapping,\targetSet) \),
we conclude that  \(\state\subZero \in
\viabilityKernel(\dynamicMapping,\targetSet)\).
\medskip

This ends the proof.
\end{proof}

\section{Application to viable control of the Wolbachia bacterium}
\label{Application_to_viable_control_of_the_Wolbachia_bacterium}

In this Section, we apply the result established in
Theorem~\ref{theo:viability2} 
to a problem related to epidemic control 
by means of biocontrol of the mosquito dengue vector.


The mosquito species \emph{Aedes aegypti} is the main transmitter of dengue.
When these mosquitoes are   infected with \emph{Wolbachia} bacterium, they
become less capable of transmitting the dengue virus to human hosts. Thanks to this
discovery, Wolbachia-based biocontrol is accepted as an ecologically
friendly and potentially cost-effective method for prevention and control of
dengue and other arboviral infections. 
We introduce now a model borrowed from \cite{bliman:hal-01579477}  representing the dynamics of a mosquito population infected
with Wolbachia. This model is described by four state variables 
\begin{equation*}
\state=(\NonWLarva,\NonWAdult,\WolbLarva,\WolbAdult) \in \RR^4
\eqfinv
\end{equation*}
where \(\NonWLarva\) and \(\NonWAdult\) represent the  uninfected mosquitoes
abundances (larva and adults respectively), 
whereas \(\WolbLarva\) and \(\WolbAdult\)  are
the infected (with Wolbachia) mosquitoes abundances (larva and adults
respectively). 
The population dynamics model is described by the following system
of differential equations
\begin{subequations}
\label{eq:WolbachiaNoControl}
\begin{align}
\dot \NonWLarva 
& = 
\alpha_U \NonWAdult
  \frac{\NonWAdult}{\NonWAdult+\WolbAdult} - \nu \NonWLarva - \mu\left(1 +
  k\left(\NonWLarva+\WolbLarva \right) \right)\NonWLarva 
\eqfinv   
\\
\dot \NonWAdult
&= 
  \nu \NonWLarva - \mu_U \NonWAdult 
\eqfinv   
\\
\dot \WolbLarva
&= 
\alpha_W \WolbAdult - \nu
  \WolbLarva - \mu\left(1 + k\left(\NonWLarva+\WolbLarva \right)
  \right)\WolbLarva 
\eqfinv   
\\
\dot \WolbAdult
&= 
\nu \WolbLarva - \mu_W \WolbAdult 
\eqfinv   
\end{align}
\end{subequations}
where all parameters are assumed to be positive \cite{bliman:hal-01579477}.   


In biocontrol, one can choose the quantity of mosquitoes infected with Wolbachia
larvae to be introduced \cite{Cali:2020}. 
This is why,  in the context of the model~\eqref{eq:WolbachiaNoControl}, we consider the
 control variable 
 \begin{equation*}
\control \in \CONTROL = [0,\control^{\sharp}] \subset \RR 
\eqfinv
 \end{equation*}
where \( \control^{\sharp} > 0 \) is the maximal quantity of mosquitoes infected with Wolbachia
larvae that can be introduced.
Then, by~\eqref{eq:WolbachiaNoControl},
we obtain a controlled dynamics which reads as~\eqref{eq:controlledModel}
with 
 \begin{equation}
\dynamicMapping(\state,\control)=
\bp{F_L(\state), F_A(\state),  G_L(\state) + \control, G_A(\state)}
\eqsepv \forall \state \in \STATE=\RR_+^4 
\eqsepv \forall \control \in \CONTROL 
\eqfinv
\label{eq:WolbachiaControlSys}
 \end{equation} 
where 
\begin{subequations}
\label{eq:WolbachiaNoControl_bis}
\begin{align}
F_L(\NonWLarva,\NonWAdult,\WolbLarva,\WolbAdult)
& = \label{eq:WolbachiaNoControl_bis_F_L}\\
& \alpha_U \NonWAdult
  \frac{\NonWAdult}{\NonWAdult+\WolbAdult} - \nu \NonWLarva - \mu\left(1 +
  k\left(\NonWLarva+\WolbLarva \right) \right)\NonWLarva 
\eqfinv   
\nonumber 
\\
F_A(\NonWLarva,\NonWAdult,\WolbLarva,\WolbAdult)
&= 
  \nu \NonWLarva - \mu_U \NonWAdult 
\eqfinv   
\\
G_L(\NonWLarva,\NonWAdult,\WolbLarva,\WolbAdult)
&= 
\alpha_W \WolbAdult - \nu
  \WolbLarva - \mu\left(1 + k\left(\NonWLarva+\WolbLarva \right)
  \right)\WolbLarva 
\eqfinv   
\\
G_A(\NonWLarva,\NonWAdult,\WolbLarva,\WolbAdult)
&= 
\nu \WolbLarva - \mu_W \WolbAdult 
\eqfinp  
\end{align}
\end{subequations}
By~\eqref{eq:WolbachiaControlSys} and~\eqref{eq:WolbachiaNoControl_bis},
the mapping~\(\dynamicMapping\) is well defined on \(\STATE=\RR_+^4\),
except for points where \(\NonWAdult=\WolbAdult=0\).
But, from the expression~\eqref{eq:WolbachiaNoControl_bis_F_L} of the first component~\(F_L \)
of \(\dynamicMapping(\cdot,\control)\), 
the mapping~\(\dynamicMapping\) can be defined in such points by continuity.


We take the stand that one of the objectives of biocontrol is to keep the
population of infected mosquitoes with Wolbachia above some thresholds
(see \cite{bliman:hal-01579477,Cali:2020} and the references therein).  
In this context, we consider positive upper population levels
$(\overline{\WolbLarva},\overline{\WolbAdult})$ 
and positive lower population levels
$(\underline{\NonWLarva},\underline{\NonWAdult})$. 
Our aim is to have the (Wolbachia) infected population of mosquitoes to
be above both $\overline{\WolbAdult},\overline{\WolbLarva}$, and the 
uninfected population to be below both
$\underline{\NonWLarva},\underline{\NonWAdult}$, permanently. 
Thus, we define the following desirable set 
\begin{equation}
\begin{split}
\targetSet = &
\left\{(\NonWLarva,\NonWAdult,\WolbLarva,\WolbAdult,\control)\in
\RR_+ \times \RR_+ \times \RR_+ \times \RR_+ \times
[0,\control^{\sharp}] \mid \right.
\\ 
& \left. \NonWLarva \leq \underline{\NonWLarva}, \NonWAdult \leq 
\underline{\NonWAdult}, \WolbLarva \geq \overline{\WolbLarva}, \WolbAdult \geq 
\overline{\WolbAdult} \right\} 
\eqfinp
\end{split}
\label{eq:WolbachiaTargetSet}
\end{equation}

 \begin{proposition}
Let the controlled dynamics mapping~\( \dynamicMapping^{\sharp}:\RR_+^4 \times [0,\control^{\sharp}] \to \RR^{4} \) 
be defined from the controlled dynamics~\eqref{eq:WolbachiaControlSys} by 
\begin{equation}
\dynamicMapping^{\sharp}(\state,\control)=
\dynamicMapping(\state,\control^{\sharp})
\eqsepv \forall \np{\state,\control} \in \RR_+^4 \times [0,\control^{\sharp}] 
\eqfinp
\label{eq:reduced_WolbachiaControlSys}
\end{equation}
Then, the viability kernels associated with the desirable set~$\targetSet$ 
and with either~$\dynamicMapping$ or
$\dynamicMapping^{\sharp}$ for the controlled dynamics coincide, that is, 
\begin{equation}
    \mathbb{V}(\dynamicMapping,\targetSet) = 
 \mathbb{V}(\dynamicMapping^{\sharp},\targetSet) 
\eqfinp
\label{eq:resultW}
\end{equation}
\label{pr:AssumptionWolbachia}
 \end{proposition}

The advantage of the equality~\eqref{eq:resultW} 
over the definition~\eqref{eq:viability_kernel} of the 
viability kernel~\( \viabilityKernel(\dynamicMapping,\targetSet) \)
is that
\(\dynamicMapping^{\sharp}\) in~\eqref{eq:reduced_WolbachiaControlSys} is not really a
controlled dynamics, as it does not depend on the control~$\control$.
In other words, an initial condition
\(\state_{0}=(\NonWLarva,\NonWAdult,\WolbLarva,\WolbLarva)_{0}\)
belongs to the viability kernel  \(\mathbb{V}(\dynamicMapping,\targetSet)\) if
and only if, using the stationary control~\(\control^{\sharp}\) 
in the differential equation 
\( \dot{\state}(\variableTime) = 
\dynamicMapping\bp{\state(\variableTime),\control^{\sharp}} \), the state and
control \((\NonWLarva(t),\NonWAdult(t),\WolbLarva(t),\WolbAdult(t),\control^{\sharp})\)
lies in \(\targetSet\), defined in~\eqref{eq:WolbachiaTargetSet}, for all $t \in \DomainTime$. Hence, the problem has been
reduced to compute the viability kernel for a single constant control policy, instead of
a family of controls, which is a far more easier problem to handle than the 
original problem.

 \begin{proof}
 The proof consists in applying Theorem~\ref{theo:viability2}. 
In order to ensure that all assumptions are satisfied, we divide the proof in three parts.
  \medskip
 
\noindent $\bullet$ 
First, we prove that the mapping~\(\dynamicMapping\) given
by~\eqref{eq:WolbachiaControlSys} is a controlled dynamics 
as in Definition~\ref{de:controlled dynamics}.
Indeed, on the one hand, by~\eqref{eq:WolbachiaControlSys} and~\eqref{eq:WolbachiaNoControl_bis},
it is straightforward that \(\dynamicMapping(\cdot,\control)\) is locally
Lipschitz on \(\STATE=\RR_+^4\), with Lipschitz constant independent of~\( \control\). 
On the other hand, it is proved in \cite[Theorem 1]{bliman:hal-01579477} that,
for all initial condition with nonnegative  components
\(\state_0 \in \RR^4_{+}\), the solutions of the controlled system
\( \dot{\state}(\variableTime) = 
\dynamicMapping\bp{\state(\variableTime),\control(\variableTime)} \),
where \( \state(0)=\state_0 \in \RR^4_{+}\) and where 
\( \control\np{\cdot}:\DomainTime\to [0,\control^{\sharp}] \)
is a measurable control path, remain in \(\RR^4_{+}\),
and that the solution is defined for all time~\( t \in [0,+\infty) \).
 \medskip
 
\noindent $\bullet$ 
Second, we show that the controlled dynamics \(\dynamicMapping\)  
is \( \np{\cone,\phi} \)-quasimonotone reducible according to Definition~\ref{de:quasimonotone_controlled_dynamics}, 
for a suitable cone $\cone \subset \RR^n$ and mapping \( \phi :
[0,\control^{\sharp}] \to [0,\control^{\sharp}] \).

On the one hand, we define the cone
\begin{equation}
\cone = \RR_- \times \RR_- \times \RR_+ \times \RR_+
\eqfinv 
\label{eq:WolbachiaCone}
\end{equation}
and the associated preorder given by, for any two vectors $\state
= (\state_1,\state_2,\state_3,\state_4)$, $\statebis =
(\statebis_1,\statebis_2,\statebis_3,\statebis_4)$,
$\state \preceq_\cone \statebis$ if and only if $\state_1 \geq \statebis_1$,
$\state_2 \geq \statebis_2$, $\state_3 \leq \statebis_3$,
$\state_4 \leq \statebis_4$. 
As the cone \(\cone\) in~\eqref{eq:WolbachiaCone} is one of the orthants of $\RR^4$,
we deduce from Proposition~\ref{prop:orthants} that,
for any measurable control path
\( \control\np{\cdot}:\DomainTime\to [0,\control^{\sharp}] \),
the mapping
${\dynamicMappingTer}_{\control\np{\cdot}}(\state,\variableTime) 
= \dynamicMapping(\state,\control(\variableTime))$ 
is $\cone$-quasimonotone if and only if
\begin{itemize}
    \item[(a)] $\frac{\partial F_L}{\partial \NonWAdult} \geq 0$, $\frac{\partial F_L}{\partial \WolbLarva} \leq 0$, $\frac{\partial F_L}{\partial \WolbAdult} \leq 0 \eqfinv$
    \item[(b)]  $\frac{\partial F_A}{\partial \NonWLarva} \geq 0$, $\frac{\partial F_A}{\partial \WolbLarva} \leq 0$, $\frac{\partial F_A}{\partial \WolbAdult} \leq 0 \eqfinv$
    \item[(c)]  $\frac{\partial G_L}{\partial \WolbAdult} \geq 0$, $\frac{\partial G_L}{\partial \NonWLarva} \leq 0$, $\frac{\partial G_L}{\partial \NonWAdult} \leq 0 \eqfinv$
    \item[(d)]  $\frac{\partial G_A}{\partial \WolbLarva} \geq 0$, $\frac{\partial G_A}{\partial \NonWLarva} \leq 0$, $\frac{\partial G_A}{\partial \NonWAdult} \leq 0 \eqfinp$
\end{itemize}
Now, these inequalities can easily be verified for the functions~\(F_L \), \(F_A \),  \(G_L \),
and \(G_A \) defined in~\eqref{eq:WolbachiaNoControl_bis}. 
Therefore, the controlled dynamics~\(\dynamicMapping\)
in~\eqref{eq:WolbachiaControlSys} 
satisfies assumption~\ref{it:quasimonotone_controlled_dynamics_a} in
Definition~\ref{de:quasimonotone_controlled_dynamics}.

On the other hand, we define the mapping 
$\phi:[0,\control^{\sharp}]\to[0,\control^{\sharp}]$ by
\begin{equation}
\phi(\control) = \control^{\sharp} \eqsepv
\forall \control \in [0,\control^{\sharp}]
\eqfinp
\label{eq:Wolbachia-reduction}  
\end{equation}
Then, we observe that, by~\eqref{eq:WolbachiaControlSys}, one has,
for all \( \control \in  [0,\control^{\sharp}] \), 
\[
 \dynamicMapping\bp{\state,\phi(\control)}-\dynamicMapping(\state,\control) 
= \np{0,0,\control^{\sharp}-\control,0}
\in \RR_- \times \RR_- \times \RR_+ \times \RR_+ =\cone 
\eqfinp
\]
By definition~\eqref{eq:K-order} of the preorder~$\preceq_\cone$
and by expression~\eqref{eq:WolbachiaCone} of the cone~\(\cone\),
we get that
\begin{equation*}
	\dynamicMapping(\state,\control) \preceq_\cone 
\dynamicMapping\bp{\state,\phi(\control)} = 
\dynamicMapping(\state,\control^{\sharp}) 
\eqsepv \forall \np{\state,\control} \in \RR_+^4 \times [0,\control^{\sharp}]
\eqfinp
\end{equation*}
Thus, the mapping~$\phi$ in~\eqref{eq:Wolbachia-reduction} 
is a $\cone$-reduction for the controlled dynamics~\(\dynamicMapping\), 
and condition~\ref{it:quasimonotone_controlled_dynamics_b} in
Definition~\ref{de:quasimonotone_controlled_dynamics} is satisfied. 
 \medskip
 
\noindent $\bullet$
Third, we prove~\eqref{eq:resultW}.

On the one hand, 
the new reduced controlled dynamics~\eqref{eq:reduced_controlled_dynamics}
satisfies \(\dynamicMapping_\phi= \dynamicMapping^{\sharp}\),
because of the expression~\eqref{eq:reduced_WolbachiaControlSys}
of~\(\dynamicMapping^{\sharp}\)
and by \( \phi(\control) = \control^{\sharp} \)
in~\eqref{eq:Wolbachia-reduction}.
On the other hand, the desirable set~\(\targetSet\)
in~\eqref{eq:WolbachiaTargetSet} has the expression 
\begin{equation}
 \targetSet = \bp{ \np{ \underline{\NonWLarva}, \underline{\NonWAdult}, 
\overline{\WolbLarva}, \overline{\WolbAdult} } + \cone } \times [0,\control^{\sharp}] 
\eqfinp 
\label{eq:WolbachiaTargetSet_bis}  
\end{equation}
We deduce that the new extended desirable set 
in~\eqref{eq:newTargetSet} satisfies 
\[
\targetSet_\cone = \targetSet + (\cone \times \{0\})=
\bp{ \np{ \underline{\NonWLarva}, \underline{\NonWAdult}, 
\overline{\WolbLarva}, \overline{\WolbAdult} } + \cone + \cone } \times 
\np{ [0,\control^{\sharp}] +0 } 
= \targetSet
\eqfinv 
\]
where we have used the property that \( \cone + \cone = \cone \),
as the cone~$\cone$ is convex and contains~$0$.
There remains to check that~\eqref{eq:equality_of_viability_kernels_assumption}
holds true. Now, by~\eqref{eq:Wolbachia-reduction}  and~\eqref{eq:WolbachiaTargetSet_bis}, we have 
\[
\bigcup_{(\state,\control) \in \targetSet} (\state+\cone)\times \phi(\control) 
=
\bp{ \np{ \underline{\NonWLarva}, \underline{\NonWAdult}, 
\overline{\WolbLarva}, \overline{\WolbAdult} } + \cone }
\times \{ \control^{\sharp} \} \subset \targetSet
\eqfinp 
\]
Therefore, we apply Theorem~\ref{theo:viability2}
and we obtain that 
\begin{equation*}
    \mathbb{V}(\dynamicMapping,\targetSet) = 
\mathbb{V}(\dynamicMapping_\phi,\targetSet_\cone) = 
\mathbb{V}(\dynamicMapping_\phi,\targetSet) =\mathbb{V}(\dynamicMapping^\sharp,\targetSet) 
\eqfinp
\end{equation*}

This ends the proof.
 \end{proof}


\paragraph{Acknowledgments.}
The authors thank the international program MATH AmSud
(\emph{MOVECO: Modeling, Optimization and Viability for Epidemics COntrol},
project 18-MATH-05)
that offered financial support for missions, together with 
\'Ecole des Ponts ParisTech (France). 
The second author was partially supported by Fondecyt N~1200355 
and by Basal Program CMM-AFB 170001, both programs from ANID-Chile.   
The  third  author  was  also  funded  by  the  program Conicyt PFCHA/Doctorado Becas Chile/2017-21171813.
 
\appendix 

\section{Comparison lemma for flows}
\label{Appendix}

We prove a lemma regarding the comparison, via a conic preorder, of flows generated by two dynamics. 

\begin{lemma}
\label{lem:comparison}
Let \( \STATE \subset \RR^{n} \) be a closed subset of~\( \RR^{n} \),
and $\dynamicMappingBis, \dynamicMappingTer:
\STATE \times \DomainTime \to \RR^{n}$ be two mappings that are 
locally Lipschitz in the first variable and measurable in the second variable,
and such that the two differential equations 
\begin{equation*}
    \dot{\state} = \dynamicMappingBis(\state,\variableTime) 
\eqsepv 
   \dot{\state} = \dynamicMappingTer(\state,\variableTime) 
\eqsepv \state(0)=\state_0 
\end{equation*}
have unique solutions, for all time \( \variableTime \in \DomainTime \)
and for all state \( \state_0 \in \STATE \), denoted by 
\( \flow{\variableTime}{\dynamicMappingBis }{\state_0} \)
and
\( \flow{\variableTime}{\dynamicMappingTer}{\state_0} \).
The mappings \( \Flow{\dynamicMappingBis} \)
and \( \Flow{\dynamicMappingTer} \) are called \emph{flows}. 

Let $\cone \subsetneq \RR^{n}$ be a closed convex cone with nonempty interior.
Suppose that 
\begin{itemize}
\item 
one of the two mappings, either $\dynamicMappingBis$ or $\dynamicMappingTer$, is
$\cone$-quasimonotone (as in Definition~\ref{de:K-quasimonotone_dynamics}), 
\item 
we have that $\dynamicMappingBis(\state,\variableTime) \preceq_\cone \dynamicMappingTer(\state,\variableTime)$,
for all \( (\state,\variableTime)\in\STATE\times\DomainTime \).
\end{itemize}
Then, the two flows \( \Flow{\dynamicMappingBis} \)
and \( \Flow{\dynamicMappingTer} \) have the following property: 
\begin{equation}
  \begin{split}
\state\subZero,\statebis\subZero \in \STATE
\text{ and }
\state\subZero \preceq_\cone \statebis\subZero \implies \\
\flow{\variableTime}{\dynamicMappingBis }{\state\subZero} 
\preceq_\cone \flow{\variableTime}{\dynamicMappingTer}{\statebis\subZero}
\eqsepv \forall \variableTime \in \DomainTime 
\eqfinp
  \end{split}
\label{eq:monotone-flows}
\end{equation}
\end{lemma}

This Lemma is a generalization of Theorem~1.1 in \cite{Hirsch2003}, 
where the implication~\eqref{eq:monotone-flows} is established in the particular
case where $\dynamicMappingBis = \dynamicMappingTer$
(when $\dynamicMappingBis = \dynamicMappingTer$, it is also proven in
\cite{Hirsch2003} that \eqref{eq:monotone-flows} is a sufficient 
condition for the $\cone$-quasimonotonicity of \(\dynamicMappingBis\)).

\begin{proof}
Observe that, for any initial conditions \(\state\subZero\) and
\(\statebis\subZero\) in~\( \STATE \), 
the solutions  \(\state(t)= \flow{\variableTime}{\dynamicMappingBis }{\state\subZero} 
\) and \(y(t)=\flow{\variableTime}{\dynamicMappingTer}{\statebis\subZero}\) 
satisfy
\begin{equation*}\label{eq:flows_exp}
\state(t)= \state\subZero + \int_0^t \dynamicMappingBis(\state(s),s)ds \eqsepv
y(t)= \statebis\subZero + \int_0^t \dynamicMappingTer(y(s),s)ds 
\eqsepv \forall t \in \DomainTime 
\eqfinp
\end{equation*}
As the mappings $\dynamicMappingBis$ and $\dynamicMappingTer$ 
are locally Lipschitz in the first variable, 
we obtain that \(\flow{\cdot}{\dynamicMappingBis }{\cdot} \) 
and \(\flow{\cdot}{\dynamicMappingTer}{\cdot}\) are continuous
in the couple argument.

As the closed convex cone $\cone \subsetneq \RR^n$ 
has nonempty interior~\( \interior\cone \), 
we introduce the following notation
\begin{equation}
  \state \ll_\cone \statebis  \Leftrightarrow 
\statebis - \state \in \interior\cone 
\eqfinp
\label{eq:strict_preorder}
\end{equation}
The relation~\( \ll_\cone \) is transitive (as \( \interior\cone +
\interior\cone \subset \interior\cone \)), but not necessarily
reflexive (as $0$ may or may not be in~\( \interior\cone \)).
The following result is established in \cite[Proposition 3.1]{Hirsch2004}
\begin{equation}
 \state \in \interior\cone 
\Leftrightarrow 
\state \in \cone \text{ and }
\proscal{\state}{\dual} >0 \eqsepv \forall \dual \in \polarCone\backslash\{0\} 
\eqfinv
\label{eq:empty_interior}
\end{equation}
where the dual cone~\( \polarCone \) 
has been defined in~\eqref{eq:positive_polar_cone}.
As a consequence, if $\state \in 
\partial\cone = \cone\backslash\interior\cone$, 
then there exists an element $\dual \in
\polarCone\backslash\{0\}$ such that  $\proscal{\state}{\dual} = 0$
(indeed, \( \polarCone\backslash\{0\} \neq \emptyset \) because of the 
assumption that $\cone \subsetneq \RR^{n}$, hence $\cone \neq \RR^{n}$).
\medskip

We assume that $\dynamicMappingBis$ is $\cone$-quasimonotone. 
In the case where $\dynamicMappingTer$ is $\cone$-quasimonotone, 
the proof is the same. 
\medskip

\noindent $\bullet$
First, we prove that, if 
\(
\dynamicMappingBis(\state,\variableTime)\ll_\cone\dynamicMappingTer(\state,\variableTime)
\), 
$\forall (\state,\variableTime) \in \STATE\times\DomainTime$, then
\begin{equation}
\state\subZero \ll_\cone \statebis\subZero
\implies
\flow{\variableTime}{\dynamicMappingBis }{\state\subZero} 
\ll_\cone \flow{\variableTime}{\dynamicMappingTer}{\statebis\subZero}
\eqsepv \forall \variableTime \in \DomainTime 
\eqfinp
\label{eq:propertyDominatedMapping2}
\end{equation}
Indeed, let us assume that this is not the case. 
Then, there would exist initial conditions \( \state\subZero \)
and \( \statebis\subZero \) in~$\STATE$, and 
$\variableS\in\DomainTime$, $\variableS >0$, such that
\begin{equation*}
 \flow{\variableTime}{\dynamicMappingBis }{\state\subZero} 
\ll_\cone \flow{\variableTime}{\dynamicMappingTer}{\statebis\subZero}
\eqsepv 
\forall \variableTime \in [0,\variableS) 
\text{ and }
 \flow{\variableS}{\dynamicMappingTer}{\statebis\subZero}
\nll_\cone \flow{\variableS}{\dynamicMappingBis}{\state\subZero} 
\eqfinv
\label{eq:inclusion_s}
\end{equation*}
that is,
\( 
 \flow{\variableTime}{\dynamicMappingBis }{\state\subZero} 
-\flow{\variableTime}{\dynamicMappingTer}{\statebis\subZero}
\in \interior\cone \), 
\( \forall \variableTime \in [0,\variableS)  \), 
and 
\( \flow{\variableS}{\dynamicMappingTer}{\statebis\subZero}-
\flow{\variableS}{\dynamicMappingBis}{\state\subZero} 
\not\in \interior\cone \). 
Since \( \cone \) is closed and 
the flows are continuous in their two arguments,
we would deduce that \( \flow{\variableS}{\dynamicMappingTer}{\statebis\subZero}
-\flow{\variableS}{\dynamicMappingBis}{\state\subZero} 
\in \cone\backslash\interior\cone = \partial\cone \).
By~\eqref{eq:empty_interior}, there would exist $\dual \in \polarCone
\backslash \{0\}$
such that both
\( \proscal{\flow{\variableS}{\dynamicMappingTer}{\statebis\subZero}
-\flow{\variableS}{\dynamicMappingBis}{\state\subZero}}{\dual} =0 \),
and 
\( \proscal{\flow{\variableTime}{\dynamicMappingTer}{\statebis\subZero}
-\flow{\variableTime}{\dynamicMappingBis}{\state\subZero}}{\dual}>0 \),
for $0\leq\variableTime<\variableS$, giving thus
\[
\derivativeTime\proscal{\flow{\variableTime}{\dynamicMappingTer}{\statebis\subZero}
-\flow{\variableTime}{\dynamicMappingBis}{\state\subZero}}{\dual}\vert_{\variableTime=\variableS}
\leq 0 
\eqfinp
\]
From the definition of the flows, we would finally obtain that 
\begin{equation}\label{eq:flows2}
\proscal{\dynamicMappingTer(\flow{\variableS}{\dynamicMappingTer}{\statebis\subZero},\variableS)}{\dual}
\leq
\proscal{\dynamicMappingBis(\flow{\variableS}{\dynamicMappingBis}{\state\subZero},\variableS)}{\dual} 
\eqfinp
\end{equation}
As we have seen that 
\(\proscal{\flow{\variableS}{\dynamicMappingTer}{\statebis\subZero}
-\flow{\variableS}{\dynamicMappingBis}{\state\subZero}}{\dual}=0 \),
and 
\( \flow{\variableS}{\dynamicMappingTer}{\statebis\subZero}
-\flow{\variableS}{\dynamicMappingBis}{\state\subZero} \in \cone \),
we would deduce that 
\( \flow{\variableS}{\dynamicMappingTer}{\statebis\subZero}
-\flow{\variableS}{\dynamicMappingBis}{\state\subZero} 
\in \cone \cap \left\{\dual \right\}\orthogonal \),
that is, \\
\( \flow{\variableS}{\dynamicMappingBis}{\state\subZero} 
\preceq_{\cone\cap \left\{\dual \right\}\orthogonal} 
\flow{\variableS}{\dynamicMappingTer}{\statebis\subZero} \)
by definition~\eqref{eq:K-order} of the preorder~\( \preceq_{\cone\cap
  \left\{\dual \right\}\orthogonal} \).
Now, since $\dynamicMappingBis$ is $\cone$-quasimonotone, 
we would deduce from~\eqref{eq:quasimonotoneCondition} that 
\begin{equation} 
\label{eq:QMTraj1}
\proscal{\dynamicMappingBis(\flow{\variableS}{\dynamicMappingBis}{\state\subZero},\variableS)}{\dual} 
\leq 
\proscal{\dynamicMappingBis(\flow{\variableS}{\dynamicMappingTer}{\statebis\subZero},\variableS)}{\dual} 
\eqfinp
\end{equation}
Combining Inequalities~\eqref{eq:flows2} and \eqref{eq:QMTraj1} would give
\[
\proscal{\dynamicMappingTer(\flow{\variableS}{\dynamicMappingTer}{\statebis\subZero},\variableS)}{\dual}
\leq
\proscal{\dynamicMappingBis(\flow{\variableS}{\dynamicMappingTer}{\statebis\subZero},\variableS)}{\dual} 
\eqfinp
\]
Now, this would contradict the assumption that 
$\dynamicMappingBis(\state,\variableTime)\ll_\cone\dynamicMappingTer(\state,\variableTime)$,
$\forall (\state,\variableTime)\in\STATE\times\DomainTime$, which indeed
implies that 
\(
\dynamicMappingTer(\flow{\variableS}{\dynamicMappingTer}{\statebis\subZero},\variableS)
-\dynamicMappingBis(\flow{\variableS}{\dynamicMappingTer}{\statebis\subZero},\variableS)
\in \interior\cone \), and, by~\eqref{eq:empty_interior}, that 
\[
\proscal{\dynamicMappingBis(\flow{\variableS}{\dynamicMappingTer}{\statebis\subZero},\variableS)}{\dual}
<
\proscal{\dynamicMappingTer(\flow{\variableS}{\dynamicMappingTer}{\statebis\subZero},\variableS)}{\dual} 
\eqfinp
\]
Therefore, the implication~\eqref{eq:propertyDominatedMapping2}
holds true.
\medskip

\noindent $\bullet$
Second, we prove~\eqref{eq:monotone-flows}.

For this purpose, we consider
\( \state\subZero,\statebis\subZero \in \STATE \) such that
\( \state\subZero \preceq_\cone \statebis\subZero \).
Then, we take \( \anyVector \in \interior\cone \neq \emptyset \)
and, for any $\epsilon>0$, we consider the following differential equation 
\begin{equation*} 
\dot{\state} = \dynamicMappingTer_\epsilon(\state,\variableTime) \defined
\dynamicMappingTer(\state,\variableTime) + \epsilon\anyVector \eqsepv
\state(0) = \statebis\subZero + \epsilon\anyVector 
\eqfinp
\end{equation*}
From the assumptions on the dynamics mapping~\(\dynamicMappingTer\), 
the above differential equation has a unique solution $\state_\epsilon(t)=
\flow{\variableTime}{\dynamicMappingTer_{\epsilon}}{\statebis\subZero +
  \epsilon\anyVector}$, defined for all $t \in \DomainTime$, and which satisfies
  \begin{equation}
  \state_\epsilon(t) = \statebis\subZero+ (1+t)\epsilon\anyVector + 
\int_0^t \dynamicMappingTer(\state_\epsilon(s),s)ds 
\eqsepv \forall t \in \DomainTime 
\eqfinp 
\label{eq:epsilon-flow}
\end{equation}
By an easy adaptation of the classical proof that solutions of
ordinary differential equations continuously depend on a 
continuous parameter (see for instance \cite{hale:1980}), 
we get the following result:
for every $t \ge 0$, 
we have $\state_\epsilon(t) \to \state(t)$ when $\epsilon \downarrow 0$, 
where $\state(\cdot)$ is solution of the differential equation
\( \dot{\state} = \dynamicMappingTer(\state,\variableTime) \),
\( \state(0) = \statebis\subZero \), that is, 
\( \state_\epsilon(t) \to
\flow{\variableTime}{\dynamicMappingTer}{\statebis\subZero}\) when $\epsilon \downarrow 0$.

Now, since \(\state\subZero \preceq_\cone \statebis\subZero\) 
and \(\dynamicMappingBis(\state,\variableTime) \preceq_\cone 
\dynamicMappingTer(\state,\variableTime)\), for all $(\state,t) \in \RR^n
\times\DomainTime$, 
we obtain that \( \state\subZero \ll_\cone \statebis\subZero + \epsilon\anyVector \)
and \( \dynamicMappingBis(\state,\variableTime)
\ll_\cone 
\dynamicMappingTer_\epsilon(\state,\variableTime)
\), $\forall (\state,\variableTime) \in \STATE\times \DomainTime$,
by the definition~\eqref{eq:strict_preorder} of the relation~\( \ll_\cone \),
where we have used that \(\cone + \interior\cone \subset  \interior\cone \) and 
\(\epsilon ( \interior\cone ) \subset  \interior\cone\), for all \(\epsilon >0\).
Thus, we can apply the implication~\eqref{eq:propertyDominatedMapping2}
established in the first part of the proof, and get 
\[
\flow{\variableTime}{\dynamicMappingBis }{\state\subZero} 
\ll_\cone 
\flow{\variableTime}{\dynamicMappingTer_{\epsilon}}{\statebis\subZero 
+ \epsilon\anyVector}=\state_\epsilon(t) \eqsepv \forall \variableTime \in
\DomainTime 
\eqfinv
\]
where \(\state_\epsilon(t)\) is given by~\eqref{eq:epsilon-flow}.
Since \(\state_\epsilon(t) \to
\flow{\variableTime}{\dynamicMappingTer}{\statebis\subZero}\) 
when $\epsilon  \downarrow 0$, for all \(t \in \DomainTime\), and 
since the cone~\( \cone \) is closed, we finally get that 
\[
\flow{\variableTime}{\dynamicMappingBis }{\state\subZero} 
\preceq_\cone \flow{\variableTime}{\dynamicMappingTer}{\statebis\subZero}
\eqsepv \forall \variableTime \in \DomainTime 
\eqfinv
\]
which is the desired result~\eqref{eq:monotone-flows}. 

\end{proof}

\end{document}